\documentclass[12pt]{amsart}
\usepackage{amsmath,amssymb,amsthm}
\usepackage{mathrsfs}
\usepackage{geometry}
\newtheorem{thm}{Theorem}[section]
\newtheorem{lem}[thm]{Lemma}
\newtheorem{prop}[thm]{Proposition}
\newtheorem{cor}[thm]{Corollary}
\theoremstyle{definition}
\newtheorem{defn}[thm]{Definition}
\newtheorem{example}[thm]{Example}

\newcommand{\e}{\varepsilon}

\newcommand{\til}{\widetilde}

\newcommand{\ev}{\operatorname{ev}}
\newcommand{\supp}{\operatorname{supp}}

\makeatletter

\newcommand*\@bigplus[1]{\vcenter{\hbox{#1$\m@th +$}}}
\newcommand*\bigplus{%
    \DOTSB 
    \mathop{%
        \mathchoice
            {\@bigplus \huge}%
            {\@bigplus \LARGE}%
            {\@bigplus {}}%
            {\@bigplus \footnotesize}%
    }%
    \slimits@ 
}

\makeatother

\title{A tautological continuous field of Roe bimodules}

\begin{document}

\author{V. Manuilov}

\date{}

\address{Moscow Center for Fundamental and Applied Mathematics, Moscow State University,
Leninskie Gory 1, Moscow, 
119991, Russia}

\email{manuilov@mech.math.msu.su}

\thanks{The author acknowledges support by the RNF grant 25-11-00018.}

\maketitle

\begin{abstract}

We generalize the notion of a continuous field of $C^*$-algebras to that of Hilbert $C^*$-bimodules. Given a partially ordered set $P$ and a monotonically non-decreasing family of ternary rings of operators (TROs) assigned to the points of $P$, we equip $P$ with a certain zero-dimensional Hausdorff topology and use a certain compactification $\gamma P$ to get the base space for a continuous field of Hilbert $C^*$-bimodules over $\gamma P$.
 
As a motivating example, we consider the set $D(X,Y)$ of coarse equivalence classes of metrics on the disjoint union of two metric spaces, $X$ and $Y$. Each such class gives rise to a uniform Roe bimodule, a TRO linking the uniform Roe algebras of $X$ and $Y$. The resulting family of TROs is non-decreasing with respect to the natural partial order on $D(X,Y)$ and thus yields a tautological continuous field of Hilbert $C^*$-bimodules over $\gamma D(X,Y)$.

\end{abstract}

\section{Introduction}

Continuous fields and Fell bundles of Banach spaces (e.g. of $C^*$-algebras, Hilbert $C^*$-modules, etc.) are a way to organize a family of these spaces over a base set into an algebraic structure with some continuity properties \cite{Exel,Dixmier}. For continuous fields the base space should be locally compact Hausdorff, while for Fell bundles it should be an inverse semigroup. We were interested in an intermediate structure, where the base set is a partially ordered set with a certain topology determined by the partial order (recall that the inverse semigroup structure determines a canonical partial order). 

First, we generalize the definition of a continuous field of $C^*$-algebras to that of Hilbert $C^*$-(bi)modules. Then we introduce on a partially ordered set $P$ a topology that makes $P$ Hausdorff and zero-dimensional, but not discrete. 
As $P$ with this topology needs not be compact or even locally compact, we pass to a certain compactification $\gamma P$, compatible with the partial order, which should be the base space for continuous fields of $C^*$-algebras and of Hilbert $C^*$-bimodules.   

If a ternary ring of operators (TRO) $M_a$ is assigned to each point $a$ of a poset $P$, and if the family of these TRO is monotonely non-decreasing (i.e. $M_a\subset M_b$ if $a\leq b$) then we show that these TRO can be organized into a continuous field of Hilbert $C^*$-bimodules over $\gamma P$.   

As an example, we consider, for metric spaces $X$ and $Y$, the set $D(X,Y)$ of coarse equivalence classes of metrics on $X\sqcup Y$. A coarse equivalence class of metrics on $X\sqcup Y$ determines a uniform Roe bimodule \cite{M2} over the uniform Roe algebras of $Y$ and of $X$, and these uniform Roe bimodules, being TROs, can be organized into a `tautological' continuous field of Hilbert $C^*$-bimodules over $\gamma D(X,Y)$. 

It was shown in \cite{M1} that the set $D(X,X)$ has a natural structure of an inverse semigroup, so this example is related to the Fell bundle construction as well.


\section{Continuous fields of Hilbert $C^*$-bimodules}

Given a set $T$ and a family $\{A_t\}_{t\in T}$ of normed spaces, we write $\prod_{t\in T}A_t$ for the set of families $\{a_t\}_{t\in T}$ such that $\sup_{t\in T}\|a_t\|_t<\infty$.

Let us remind the definition of a continuous field of $C^*$-algebras (aka a $C(T)$-$C^*$-algebra, \cite{Dixmier}, Chapter 10).
Let $T$ be a compact Hausdorff space, $A$ and $A_t$, $t\in T$, $C^*$-algebras, $A\subset\prod_{t\in T}A_t$, $p_t:A\to A_t$ the canonical projections. We say that an action of $C(T)$ on $A$ is compatible with the inclusion $A\subset\prod_{t\in T}A_t$ if $p_t(fa)=f(t)p_t(a)$ for any $a\in A$, $t\in T$, $f\in C(T)$. A quadruple $\mathcal A=(T,A,\{A_t\}_{t\in T},\{p_t\}_{t\in T})$ is a continuous field of $C^*$-algebras if $p_t$ is surjective for any $t\in T$, if $C(T)$ acts on $A$ compatibly with the above inclusion, and if the map $t\mapsto\| p_t(a)\|$ is continuous for any $a\in A$.

Similarly, one can define a continuous field of Hilbert $C^*$-modules.

\begin{defn}
Let $M_t$ be a left Hilbert $C^*$-module over $A_t$ with the inner product ${}_{A_t}\langle\cdot,\cdot\rangle$, $M$ a left Hilbert $C^*$-module over $A$ with the inner product ${}_{A}\langle\cdot,\cdot\rangle$, and $M\subset\prod_{t\in T}M_t$, with the projection maps $\pi_t:M\to M_t$. A quadruple $\mathcal M=(T,M,\{M_t\}_{t\in T},\{\pi_t\}_{t\in T})$ is a continuous field of left Hilbert $C^*$-modules over $\mathcal A$ if 
\begin{enumerate}
\item
$\pi_t$ is surjective for any $t\in T$; 
\item
$C(T)$ acts on $M$ compatibly with the inclusion $M\subset\prod_{t\in T}M_t$, i.e. $\pi_t(fm)=f(t)\pi_t(m)$ for any $t\in T$, $f\in C(T)$, $m\in M$;
\item
$\pi_t(am)=p_t(a)\pi_t(m)$ for any $a\in A$, $m\in M$, $t\in T$;
\item
${}_{A_t}\langle \pi_t(m),\pi_t(n)\rangle=p_t({}_A\langle m,n\rangle)$ for any $m,n\in M$, $t\in T$. 

\end{enumerate}

Note that in this case the map $t\mapsto\|\pi_t(m)\|$ is continuous. Indeed, the map $t\mapsto\|p_t({}_A\langle m,m\rangle)\|$ is continuous as ${}_A\langle m,m\rangle\in A$, so its composition with the square root is continuous too.

\end{defn}

\begin{example}
Consider $C(T)$ as a continuous field with the fiber $\mathbb C$ over $T$. Then the standard Hilbert $C^*$-module $M=l_2(C(T))$ is a continuous field of Hilbert spaces.

\end{example}

\begin{example}
Let $\mathcal A=(T,A,\{A_t\}_{t\in T},\{p_s\}_{t\in T})$ be a continuous field of $C^*$-algebras. 
Let $M=l_2(A)$ be the standard Hilbert $C^*$-module over $A$. Then $p_s$ induces a map $\pi_t:l_2(A)\to l_2(A_t)$, which is surjective (as $p_t$ is surjective), and $(T,l_2(M),\{l_2(A_t)\}_{t\in T},\{\pi_t\}_{t\in T})$ is a continuous field of Hilbert $C^*$-modules.

Similarly, let $e\in A$ be a projection, and let $M=pA\subset A$. Then $p_t(e)\in A_t$ is a projection for any $t\in T$. Set $M_t=p_t(e)A_t\subset A_t$. Then $\pi_t=(p_t)|_{pA}$ is surjective, and $(T,M,\{M_t\}_{t\in T},\{\pi_t\}_{t\in T})$ is a continuous field of Hilbert $C^*$-modules.

\end{example}

\begin{defn}\label{def2}
Let $\mathcal A=(T,A,\{A_t\}_{t\in T},\{p_t\}_{t\in T})$ and $\mathcal B=(T,B,\{B_t\}_{t\in T},\{q_t\}_{t\in T})$ be two continuous fields of $C^*$-algebras. Then we define a continuous field of Hilbert $\mathcal A$-$\mathcal B$-$C^*$-bimodules over $T$ as a quadruple $\mathcal M=(T,M,\{M_{t}\}_{t\in T},\{\pi_t\}_{t\in T})$, where $M$ is a Hilbert $A$-$B$-$C^*$-bimodule with inner products ${}_A\langle\cdot,\cdot\rangle$ and $\langle\cdot,\cdot\rangle_B$ (resp., $M_{t}$ is a Hilbert $A_t$-$B_t$-$C^*$-bimodule with inner products ${}_{A_t}\langle\cdot,\cdot\rangle$ and $\langle\cdot,\cdot\rangle_{B_t}$), $M_{t}\subset\prod_{t\in T}M_{t}$, $C(T)$ acts on $M$ compatibly with this inclusion; the canonical projections $\pi_{t}:M\to M_{t}$ are surjective; 
\begin{equation}\label{def_eq1}
\pi_t(am)=p_t(a)\pi_t(m)\quad \mbox{and} \quad  \pi_t(ma)=\pi_t(m)q_t(a) 
\end{equation}
for any $a\in A$, $m\in M$, $t\in T$;
and 
\begin{equation}\label{def_eq2}
p_t({}_A\langle m,n\rangle)={}_{A_t}\langle \pi_{t}(m),\pi_{t}(n)\rangle \quad\mbox{and}\quad q_t(\langle m,n\rangle_B)=\langle \pi_t(m),\pi_t(n)\rangle_{B_t}
\end{equation}
for any $t\in T$, $m,n\in M$. This also implies continuity of the map $t\mapsto\|\pi_t(m)\|$ for any $m\in M$.  

\end{defn}

\begin{prop}\label{prop1}
Let $(T,M,\{M_t\}_{t\in T},\{\pi_t\}_{t\in T})$ be a continuous field of left Hilbert $C^*$-modules over a continuous field $(T,A,\{A_t\}_{t\in T},\{p_t\}_{t\in T})$ of $C^*$-algebras. Let $t_0\in T$, $\e>0$. Let $m\in M$ satisfy $p_{t_0}({}_A\langle m,m\rangle)=e\in A_{t_0}$ is a projection. Then there exists a neighborhood $U$ of $t_0$ in $S$ and $n\in M$ such that 
\begin{enumerate}
\item
$\|n-m\|<\e$ and $\pi_{t_0}(n)=\pi_{t_0}(m)$;
\item
$p_t({}_A\langle n,n\rangle)$ is a projection for any $t\in U$.

\end{enumerate}

\end{prop}
\begin{proof}
Let $g$ be a continuous function on $\mathbb R$ such that $g(x)=0$ when $x\in[-\frac{1}{4},\frac{1}{4}]$ and $g(x)=\frac{1}{x}$ when $x\in[\frac{3}{4},\frac{5}{4}]$, and let $f(x)=xg(x)$. Let $a={}_A\langle m,m\rangle\in A$. As $\|p_t(a-a^2)\|$ is small when $t$ is close to $t_0$, there exists a neighborhood $U'$ of $t_0$ such that $p_t(f(a))$ is a projection. Set $n'=g^{1/2}(a)m$. Then 
$$
{}_A\langle n',n'\rangle=g^{1/2}(a)ag^{1/2}(a)=f(a), 
$$
hence $p_t({}_A\langle n',n'\rangle)=p_t(f(a))$ is a projection when $t\in U'$. We also have 
$$
\|n'-m\|=\|(1-g^{1/2}(a))m\|\leq \|1-g^{1/2}(a)\|. 
$$
As $p_{t_0}(1-g^{1/2}(a))=0$ and $1-g^{1/2}(a)\in A$, $\|n'-m\|$ is small when $s\in U'$. Take some open set $U$ such that $t_0\in U$, $\overline{U}\subset U'$. To make $\|n-m\|$ small for any $t\in T$, cover $S$ by $U'$ and $V=S\setminus\overline{U}$, and use a partition of unity $\varphi_{U'}$, $\varphi_V$ for this cover to set $n=\varphi_{U'} n'+\varphi_V m$.
\end{proof}

\begin{prop}\label{prop2}
Let $(T,M,\{M_t\}_{t\in T},\{\pi_t\}_{t\in T})$ be a continuous field of left Hilbert $C^*$-modules over a continuous field $(T,A,\{A_t\}_{t\in T},\{p_t\}_{t\in T})$ of $C^*$-algebras. Let $t_0\in T$, $\e>0$. Let $m,n\in M$ be orthogonal over $t_0$, i.e. $p_{t_0}({}_A\langle m,n\rangle)=0$, and let $p_{t_0}({}_A\langle m,m\rangle)$, $p_{t_0}({}_A\langle n,n\rangle)$ be projections. Then there exists a neighborhood $U$ of $t_0$ in $S$ and $m',n'\in M$ such that 
\begin{enumerate}
\item
$\|m'-m\|<\e$ and $\pi_{t_0}(m')=\pi_{t_0}(m)$; $\|n'-n\|<\e$ and $\pi_{t_0}(n')=\pi_{t_0}(n)$; 
\item
$p_t({}_A\langle n',n'\rangle)$ and $p_t({}_A\langle n',n'\rangle)$ are projections for any $t\in U$;
\item
$p_t({}_A\langle m',n')=0$ for any $t\in U$.

\end{enumerate}

\end{prop}
\begin{proof}
First, apply Proposition \ref{prop1} to obtain $m'$ satisfying (i) and (ii). Then set $\tilde n=n-{}_A\langle m',n\rangle m'$. Then $\pi_{t_0}(\tilde n-n)=0$ and $p_t({}_A\langle \tilde n,m')=0$ for any $t\in T$ such that $p_t({}_A\langle m',m'\rangle)$ is a projection, so we can apply Proposition \ref{prop1} to obtain $n'$ from $\tilde n$ satisfying (i) and (ii). As $n'$ is a multiple of $\tilde n$, $n'$ and $m'$ are orthogonal over any $t\in U$.  
\end{proof}

\begin{cor}
Let $(T,M,\{M_t\}_{t\in T},\{\pi_t\}_{t\in T})$ be a continuous field of left Hilbert $C^*$-modules, $t_0\in T$. Let $M_{t_0}$ be a free finitely generated Hilbert $C^*$-module over $A_{t_0}$, i.e. $M_{t_0}\cong A_{t_0}^k$, $k\in\mathbb N$. Then there exists a neighborhood $U$ of $t_0$ in $T$ such that $M_t$ contains $A_t^k$ as an orthogonal direct summand. 

\end{cor}
\begin{proof}
Take an orthonormal basis $\{e_i\}_{i=1}^k$ in $M_{t_0}$. By definition, there exist $m_i\in M$, $i=1,\ldots,k$, such that $p_{t_0}(m_i)=e_i$. Apply Proposition \ref{prop2} $k$ times to obtain an orthonormal elements in $M_t$ for $t$ close to $t_0$. Then apply the Dupr\'e--Fillmore theorem \cite{Dupre-Fillmore} to conclude that $A_t^k$ is an orthogonal direct summand in $M_t$.

\end{proof}

\section{Continuous fields of Hilbert $C^*$-bimodules over posets}

\subsection{A topology on posets and a compactification}

Let $P$ be a partially ordered set. For $a\in P$, set 
$$
U_a=\{b\in P:a\leq b\},\quad V_a=P\setminus U_a. 
$$
Clearly, $a\in U_a$ for any $a\in P$. Consider the topology on $P$ determined by the subbase consisting of the sets $U_a$ and $V_a$, $a\in P$. This topology on posets is not so popular as some others, e.g. the interval topology, but it has two advantages: its subbase consists of clopen sets, and it has nice separation properties.

\begin{lem}
$P$ is Hausdorff and Tychonoff.

\end{lem}
\begin{proof}
Let $a\neq b\in P$. Then either $a\leq b$ or $b\leq a$ does not hold. Suppose that $b\leq a$ is wrong. Then $b\in V_a$. As $a\in U_a$ and $U_a\cap V_a=\emptyset$, $P$ is Hausdorff. 

Let $F\subset P$ be a closed subset, $a\notin F$. Then $P\setminus F$ is an open neighborhood of $a$, therefore, there exist $a_1,\ldots,a_n,a'_1,\ldots,a'_m\in P$ such that 
$$
W=U_{a_1}\cap\cdots\cap U_{a_n}\cap V_{a'_1}\cap\cdots\cap V_{a'_m}
$$ 
is a clopen neighborhood of $a$ such that $W\subset P\setminus F$. Set $f(b)=\left\lbrace\begin{array}{cl}1,&\mbox{if\ }b\in W;\\0,&\mbox{if\ }b\notin W.\end{array}\right.$ Then $f$ is continuous, $f(a)=1$ and $f|_F=0$, hence $P$ is Tychonoff.
 \end{proof}

A poset $P$ with this topology needs not to be compact. 

\begin{lem}\label{lem:noncompact} 
If $P$ has infinitely many minimal elements $c_i$, $i\in I$, then $P$ is not compact.  

\end{lem}
\begin{proof}
As $a\in U_a$ for any $a\in P$, the family $\{U_a\}_{a\in P}$ is a cover for $P$. If there exists a finite subcover $U_{a_1},\ldots, U_{a_n}$ then, for any $c_i$ there exists $j$ such that $c_i\in U_{a_j}$, i.e. $a_j\leq c_i$, hence $a_j=c_i$. But by assumption, there are infinitely many different $c_i$ --- a contradiction.  
\end{proof}

Let us define a compactification of $P$.
Let $\chi_a$ (resp., $\chi'_a$) be the characteristic function of the set $U_a$ (resp., of $V_a$), and let $C$ be the $C^*$-subalgebra of the $C^*$-algebra $C_b(P)$ of all bounded continuous functions on $P$ generated by all $\chi_a$ and $\chi'_a$, $a\in P$. Then, by Gelfand duality, $C=C(T)$ for some compact Hausdorff space $T$, and the inclusion $C(T)\subset C_b(P)$ determines a surjective map $\beta P\to T$ from the Stone--\v Cech compactification $\beta P$. As $C(T)$ separates points of $P$, $P$ embeds into $T$ as a dense subset, so $T$ is a compactification of $P$. Denote it by $\gamma P$. Clearly, $\gamma P$ is Hausdorff. 

The partial order on $P$ extends to a partial order on $\gamma P$. For $r,t\in\gamma P$ set $r\leq t$ if $\til\chi_a(r)=1$ implies $\til\chi_a(t)=1$, $a\in P$. 

If $P$ has the smallest element 0 then it suffices to use only the unit and the characteristic functions $\chi_a$ without $\chi'_a$ to obtain $C(\gamma P)$, as $\chi_0$ is the unit in $C(\gamma P)$. For simplicity's sake we shall assume here and further that $P$ has the smallest element. In this case all characteristic functions $\chi_a$ are monotonely non-decreasing, hence our compactification is a quotient of the Nachbin compactification \cite{Nachbin}, which is obtained by Gelfand duality from the $C^*$-algebra generated by all monotonely non-decreasing continuous bounded functions on $P$ and is usually used to compactify partially ordered sets. 

For $f\in C$ we write $\til f$ for its extension, by continuity, to a function on $\gamma P$, and $\til f(t)=\ev_t(\til f)$ is the evaluation of $f$ at the point $t\in\gamma P$.

\subsection{Ternary rings of operators}

Let $H_1$ and $H_2$ be Hilbert spaces, $\mathbb B(H_1,H_2)$ the set of bounded operators from $H_1$ to $H_2$. A closed linear subspace $M\subset\mathbb B(H_1,H_2)$ is a ternary ring of operators (TRO) if $MM^*M\subset M$. The closures of $A=MM^*$ and $B=M^*M$ are $C^*$-algebras, and $M$ is a Hilbert $A$-$B$-$C^*$-bimodule with the inner products ${}_A\langle m,n\rangle=mn^*$, $\langle m,n\rangle_B=m^*n$. The ternary product is given by $(mnr)=mn^*r$, $m,n,r\in M$. Our reference for ternary rings of operators is \cite{TRO}. 

Let $\{M_a\}_{a\in P}$ be a family of TRO, $M_a\subset\mathbb B(H_1,H_2)$ for any $a\in P$. We say that it is compatible with the partial order on $P$ if $a\leq b$ implies that $M_a\subset M_b$.   

Let $s\in M_a$. Set $\chi_{s,a}(b)=\left\lbrace\begin{array}{cl}s,&\mbox{if\ }a\leq b;\\0,&\mbox{otherwise.}\end{array}\right.$ 
Using the canonical embedding of the algebraic tensor product $C_b(P)\otimes\mathbb B(H_1,H_2)$ into $C_b(P,\mathbb B(H_1,H_2))$, we write $\chi_{s,a}$ as $\chi_a\otimes s$.

If $\{M_a\}_{a\in P}$ is compatible with the partial order then $\chi_{s,a}(b)\in M_a$ for any $a,b\in P$. Let $M$ be the smallest norm closed linear subspace of $C_b(P,\mathbb B(H_1,H_2))$ that contains all $\chi_{s,a}$, $a\in P$, $s\in M_a$, and is closed under ternary products and under multiplication by bounded continuous functions on $P$.

\begin{lem}
If $m\in M$ then 
\begin{enumerate}
\item
$m\in C_b(P,\mathbb B(H_1,H_2))$;
\item
$m$ extends by continuity to a map from $\gamma P$;
\item
the map $M\to M_a$, $m\mapsto \pi_a(m)$, is surjective for any $a\in P$. 

\end{enumerate}

\end{lem}
\begin{proof}
The map $\chi_{s,a}:P\to\mathbb B(H_1,H_2)$ is obviously continuous and extends to a continuous map $\til\chi_{s,a}=\til\chi_a\otimes s$ from $\gamma P$ to $\mathbb B(H_2,H_1)$. The same holds for linear combinations, for products by continuous functions on $P$, and by ternary products of the maps $\chi_{s,a}$, $a\in P$. Denote the set of all these maps by $C_0\subset C_b(P,\mathbb B(H_1,H_2))$. If a net $\{f_\lambda\}_{\lambda\in\Lambda}$ in $C_0$ converges to a map $f$ on $P$ uniformly then each $f_\lambda$ extends to a continuous map $\tilde f_\lambda$ on $\gamma P$. Then for any $\varepsilon>0$ there exists $\lambda_0\in\Lambda$ such that $\lambda,\lambda'\geq\lambda_0$ implies that $\sup_{a\in P}\|f_\lambda(a)-f_{\lambda'}(a)\|<\varepsilon$, therefore, $\{f_\lambda(t)\}_{\lambda\in\Lambda}$ is a Cauchy net for any $t\in\gamma P$, hence $\lim_{\Lambda}f_\lambda(t)$ exists in $\mathbb B(H_1,H_2)$, i.e. $f$ extends by continuity to a map on $\gamma P$.   
Finally, as $\pi_a(\chi_{s,a})=s$ for any $s\in M_a$, the map $m\mapsto\pi_a(m)$ is surjective. 
\end{proof}


We denote the continuous extension of a map $m\in M$ to $\gamma P$ by $\til m$.

Define the evaluation map $\pi_t:M\to \mathbb B(H_1,H_2)$ by $\pi_t(m)=\til m(t)$, $t\in \gamma P$. The following result is folklore.

\begin{thm}\label{surj}
The map $\pi_t:M\to\mathbb B(H_1,H_2)$ has closed range for any $t\in\gamma P$.

\end{thm}
\begin{proof}
Fix $t\in\gamma P$.
Let $m_k\in M$, $k\in\mathbb N$, $\pi_t(m_k)=s_n$, $\lim_{k\to\infty}s_k=s$. We shall construct $m\in M$ such that $\pi_t(m)=s$.

Without loss of generality, passing to a subsequence, we may assume that $\|s_{k+1}-s_k\|<2^{-k}$ for any $k\in\mathbb N$. Set $n_k=m_{k+1}-m_k\in M$.

By continuity, for each $k\in\mathbb N$ there exists an open neighborhood $W_k$ of $t$ such that $\|n_k(t)\|<2^{-k+1}$ for any $a\in W_k$.
Set $W'_1=W_1$, $W'_k=W'_{k-1}\cap W_k$, $k\geq 2$. Then $t\in\ldots\subset W'_{k+1}\subset W'_k\subset\ldots\subset W'_1$. 

For each $k\in\mathbb N$ let $\varphi_k\in C(\gamma P)$ satisfy $\varphi_k(t)=1$, $0\leq \varphi_k(r)\leq 1$ for any $r\in\gamma P$ and $\supp\varphi_k\subset W'_k$.
Set $l_1=\varphi_1 m_1$, $l_k=\varphi_k n_{k-1}$, $k\geq 2$. Then $l_k\in M$ for any $k\in\mathbb N$. As $\|l_k\|\leq\sup_{r\in\gamma P}\|n_{k-1}(r)\|<2^{-k+2}$, the series $\sum_{k\in\mathbb N}l_k$ is convergent to some $m\in M$. Let $p_N=\sum_{k=1}^N l_k$ be the partial sum for this series. Then $p_N(t)=s_N$, hence $m(t)=\lim_{N\to\infty}p_N(t)=s$. 
\end{proof}  

Denote the range of $\pi_t$ by $M_t$.

\begin{cor}
$M_t$ is a ternary ring of operators for any $t\in\gamma P$.

\end{cor}

\begin{lem}
For $t\in\gamma P$, $M_t=\overline{\bigplus_{{a\in P,}\atop{a\leq t}}M_a}$.

\end{lem}
\begin{proof}
If $t\in P$ then this is obvious, so let $t\notin P$. Take $m=\sum_{i=1}^k \til\chi_{a_i}\otimes s_i$, $a_i\in P$, $s_i\in M_{s_i}$.  Elements of this form are dense in $M$. Then $\pi_t(m)=m(t)=\sum_{i:a_i\leq t}s_i\in\bigplus_{{a\in P,}\atop{a\leq t}}M_a$. Theorem \ref{surj} finishes the proof.
\end{proof}

\subsection{Continuous field of $C^*$-bimodules}

Similarly to ternary rings of operators, one can construct from a family $\{A_u\}_{u\in P}$ of $C^*$-subalgebras of $\mathbb B(H)$ compatible with the partial order on $P$ a $C^*$-subalgebra $A$ of $C_b(P,\mathbb B(H))$ generated by the maps $\chi_u\otimes a$, $a\in A_u$. Then $A$ can be embedded into $C(\gamma P,\mathbb B(H))$, the $C^*$-algebras $A_v$ can be defined as above for $v\in\gamma P\setminus P$, and the maps $p_v:A\to A_v$ can be defined.

\begin{lem}
The quadruple $\mathcal A=(\gamma P,A,\{A_v\}_{v\in\gamma P},\{p_v\}_{v\in\gamma P})$ is a continuous field of $C^*$-algebras.

\end{lem}
\begin{proof}
It remains only to show that the map $v\mapsto \|p_v(a)\|$ is continuous for any $a\in A$, but this follows from continuity of $a$ as a map from $\gamma P$ to $\mathbb B(H)$.
\end{proof}

Let $\{M_u\}_{u\in P}$ be a family of TRO in $\mathbb B(H_1,H_2)$ compatible with the partial order on $P$, and let $\gamma P$, $\{M_v\}_{v\in\gamma P}$, $M$ and $\{\pi_v\}_{v\in\gamma P}$ be as constructed above. For each $v\in\gamma P$ set $A_v=\overline{M_vM_v^*}$, $B_v=\overline{M_v^*M_v}$. Then $\{A_u\}_{u\in P}$, $\{B_u\}_{u\in P}$ are families of $C^*$-subalgebras of $\mathbb B(H_2)$ and of $\mathbb B(H_1)$, respectively, and both are compatible with the partial order on $P$. Then they determine continuous fields 
$$
\mathcal A=(\gamma P,A,\{A_v\}_{v\in\gamma P},\{p_v\}_{v\in\gamma P})\quad\mbox{and}\quad \mathcal B=(\gamma P,B,\{B_v\}_{v\in\gamma P},\{q_v\}_{v\in\gamma P})
$$ 
of $C^*$-algebras. 

\begin{thm}
$\mathcal M=(\gamma P,M,\{M_v\}_{v\in\gamma P},\{\pi_v\}_{v\in\gamma P})$ is a continuous field of Hilbert $\mathcal A$-$\mathcal B$-$C^*$-bimodules.

\end{thm}
\begin{proof}
Surjectivity of $\pi_v$ is proved in Theorem \ref{surj}. It remains to check other (algebraic) properties of Definition \ref{def2}. As $C_b(\gamma P,\mathbb B(H_1,H_2))$ is a bimodule over $C_b(\gamma P,\mathbb B(H_2))$ and $C_b(\gamma P,\mathbb B(H_1))$, we have (\ref{def_eq1}). If we define ${}_A\langle m,n\rangle=mn^*$, $\langle m,n\rangle_B=m^*n$ then (\ref{def_eq2}) follows from the same property for $C_b(\gamma P,\mathbb B(H_1,H_2))$.  
\end{proof}

\section{Tautological continuous field of uniform Roe bimodules}

\subsection{Poset of metrics on pairs of metric spaces}

Let $X$ and $Y$ be metric spaces with metrics $d_X$ and $d_Y$, respectively. For an ordered pair $(X,Y)$, let $\mathcal D(X,Y)$ be the set of metrics $d$ on $X\sqcup Y$ such that
\begin{enumerate}
\item
$d|_X=d_X$, $d|_Y=d_Y$;
\item
$\inf_{x\in X,y\in Y}d(x,y)>0$.
\end{enumerate}

If we change the order of $X$ and $Y$ then we write $d^*$ for the same metric on $X\sqcup Y$, i.e. $d^*\in\mathcal D(Y,X)$ for $d\in\mathcal D(X,Y)$.

Recall that two metrics, $d_1,d_2\in\mathcal M(X,Y)$, are coarsely equivalent if there exists a homeomorphism $\phi$ of $[0,\infty)$ such that
$$
\phi^{-1}(d_1(x,y))\leq d_2(x,y)\leq \phi(d_1(x,y))
$$
for any $x\in X$, $y\in Y$. Denote by $D(X,Y)$ the set of coarse equivalence classes of $\mathcal D(X,Y)$.

Define a partial order on $D(X,Y)$: $[d_1]\leq [d_2]$ if there exists a homeomorphism $\phi$ of $[0,\infty)$ such that $d_2(x,y)\leq \phi(d_1(x,y))$ for any $x\in X$, $y\in Y$. 

For simplicity's sake, instead of general metric spaces, from now on we shall consider only discrete metric spaces that are uniformly discrete ($X$ is uniformly discrete if $\inf_{x,x'\in X}d_X(x,x')>0$) and of bounded geometry ($X$ has bounded geometry if for any $R>0$ there exists $C>0$ such that any ball of radius $R$ has no more than $C$ points). For a discrete space $X$ we write $H_X=l_2(X)$ for the Hilbert spaces of square-summable functions on $X$, and we fix an orthonormal basis on $H_X$ consisting of the characteristic functions of points, $\delta_x(z)=\left\lbrace\begin{array}{cl}1,&\mbox{if\ }z=x;\\0,&\mbox{otherwise,}\end{array}\right.$, $x\in X$. For an operator $S:H_X\to H_Y$, set $S_{yx}=\langle \delta_y,S\delta_x\rangle$, $x\in X$, $y\in Y$. 

A bounded operator $S:H_X\to H_Y$ has propagation $\leq L$ with respect to the metric $d\in\mathcal M(X,Y)$ if $S_{yx}\neq 0$ implies that $d(x,y)\leq L$. Let $M_{[d]}$ denote the norm closure, in $\mathbb B(H_X,H_Y)$, of the set of operators of finite propagation. This is a $C^*_u(Y)$-$C^*_u(X)$-bimodule over the uniform Roe algebras of $X$ and $Y$ \cite{M2}, and, clearly, it depends not on the chosen metric, but on its coarse equivalence class. 

\begin{lem}
$M_{[d]}$ is a ternary ring of operators.

\end{lem}
\begin{proof}
Let $T,S,R\in M_{[d]}$. We have to show that $TS^*R\in M_{[d]}$, and it suffices to do that for $T,S,R$ of a fixed propagation, so suppose that they all have propagation $\leq L$. Let $x\in X$, $y\in Y$. Suppose that $(TS^*R)_{yx}=\sum_{v\in Y,u\in X}T_{yu}\overline{S}_{vu}R_{vx}\neq 0$. Then there should exist $u\in X$ and $v\in Y$ such that $d(x,v)\leq L$, $d(u,v)\leq L$ and $d(u,y)\leq L$. Then, by the triangle inequality, $d(x,y)\leq 3L$, hence $TS^*R$ has propagation $\leq 3L$.
\end{proof}

\begin{lem}
The family $\{M_{[d]}\}_{[d]\in D(X,Y)}$ is compatible with the partial order on $D(X,Y)$.

\end{lem}
\begin{proof}
Obvious.
\end{proof} 

If $d_1\in\mathcal D(X,y)$ and $d_2\in\mathcal D(Y,Z)$ then one can define their product $d_1d_2\in\mathcal D(X,Z)$ by
\begin{equation}\label{eq:compos}
(d_2d_1)(x,z)=\inf_{y\in Y}(d_1(x,y)+d_2(y,z)),
\end{equation}
and this gives an associative multiplication $D(X,Y)\times D(Y,Z)\to D(X,Z)$.

\subsection{Multiplicativity of $M_{[d]}$}

Let $d\in\mathcal D(X,Z)$, $U\subset X$, $W\subset Z$. Recall that a bijection $t:U\to W$ is a partial translation \cite{Brodzki} if there exists $C>0$ such that $d(x,f(x))<C$ for any $x\in C$. Any partial translation $t$ determines an operator $T$ acting by $T\delta_x=\left\lbrace\begin{array}{cl}\delta_{t(x)},&\mbox{if\ }x\in U;\\0,&\mbox{if\ }x\notin U.\end{array}\right.$ It is well known that (see, e.g., Lemma 2.4 in \cite{Spakula-Willett}), as $X$, $Z$ are of bounded geometry, any finite propagation operator from $H_X$ to $H_Z$ is a finite sum of the form $\sum_i f_iT_i$ with $f_i\in C(X)$, $T_i$ operators determined by partial isometries. 

\begin{thm}
Let $d_1\in\mathcal D(X,Y)$, $d_2\in\mathcal D(Y,Z)$. Then $M_{[d_2]}M_{[d_1]}=M_{[d_2d_1]}$.

\end{thm}     
\begin{proof}
Let $S\in M_{[d_1]}$, $T\in M_{{[d_2]}}$ have finite propagation $\leq L$, and let $R=TS$. Then the propagation of $R$ is $\leq 2L$, hence $M_{[d_2]}M_{[d_1]}\subset M_{[d_2d_1]}$. 
To prove the opposite inclusion, consider a partial translation $t:U\to W$, $U\in X$, $W\in Z$, with the constant $C$, for the metric $d=d_1d_2$. Then for any $x\in U$ there exists $y\in Y$ such that $d_1(x,y)<C$ and $d_2(y,f(z))<C$. This gives a map $f:U\to Y$. Set $V=f(U)$. If $f(x_1)=f(x_2)=y$ then $d_X(x_1,x_2)\leq d(x_1,y)+d(x_2,y)<2C$, and, by the bounded geometry condition, there exists $N\in\mathbb N$ such that $f^{-1}(y)$ consists of $\leq N$ points. Therefore, $U$ can be divided into $\leq N$ subsets, $U=U_1\sqcup\ldots\sqcup U_N$ such that $f|_{U_i}$ is injective for each $i=1,\ldots,N$. Then $f|_i{U_i}$ are partial translations in $X\sqcup Y$ for the metric $d_1$. If $x\in U_i$, set $g_i(f_i(x))=t(x_i)$. Then $g_i:V\to W_i$, $W=W_1\sqcup\ldots\sqcup W_N$ are also partial translations, and $g_i(f_i(x))=t(x)$.
Let $T$, $F_i$, $G_i$ be the operators corresponding to the partial translations $t$, $f_i$, $g_i$, respectively. Then $T=\sum_i G_iF_i$. Therefore, $M_{[d]}\subset M_{[d_2]}M_{[d_1]}$.  
\end{proof}

\begin{cor}
$M_{[d^*]}M_{[d]}=M_{[d^*d]}$, $M_{[d]}M_{[d^*]}=M_{[dd^*]}$.

\end{cor}

Let $d\in\mathcal D(X,Y)$, $x_1,x_2\in X$, $y_1,y_2\in Y$. Set
$$
d^r(x_1,x_2)=\left\lbrace\begin{array}{cl}\inf_{u\in Y}(d(x_1,u)+d(x_2,u)),&\mbox{if\ }x_1\neq x_2;\\0,&\mbox{if\ } x_1=x_2,\end{array}\right. 
$$
$$
d^l(y_1,y_2)=\left\lbrace\begin{array}{cl}\inf_{u\in X}(d(u,y_1)+d(u,y_2)),&\mbox{if\ }y_1\neq y_2;\\0,&\mbox{if\ } y_1=y_2.\end{array}\right.
$$

Note that $d^*d$ is a metric in $\mathcal D(X,X)$, and $dd^*$ is a metric in $\mathcal D(Y,Y)$.

Here we have to distinguish the two copies of $X$, so we shall write $X$ for the first copy of $X$, and $X'$ for the second copy. The identity map is a homeomorphism $i:X\to X'$ in this notation. We write $x'\in X'$ for $i(x)$, $x\in X$. We also write $Y$ and $Y'$ for the two copies of $Y$.

\begin{lem}
For $d\in\mathcal D(X,Y)$, one has $(d^*d)(x_1,x'_2)=d^r(x_1,x_2)$ and $(dd^*)(y_1,y'_2)=d^l(y_1,y_2)$ for any $x_1\neq x_2$, $y_1\neq y_2$.

\end{lem}
\begin{proof}
This directly follows from (\ref{eq:compos}).
\end{proof} 

\begin{cor}
$d^r$ is a metric on $X$; $d^l$ is a metric on $Y$.

\end{cor}

Let $C^*_u(X,d')$ be the uniform Roe algebra of $X$ with respect to a metric $d'$.

\begin{cor}
Under the identification $H_X=H_{X'}$ and $H_Y=H_{Y'}$, one has $C^*_u(X,d^r)= M_{[d^*d]}$ and $C^*_u(Y,d^l)=M_{[dd^*]}$.

\end{cor} 

If $d_1\leq d_2$ then $C^*_u(X,d_1^r)\subset C^*_u(X,d_2^r)$, and $C^*_u(Y,d_1^l)\subset C^*_u(Y,d_2^l)$. 
Therefore, one can define  continuous fields of $C^*$-algebras $\mathcal A=(T,A,\{A_t\}_{t\in T},\{p_t\}_{t\in T})$ and $\mathcal B=(T,B,\{B_t\}_{t\in T},\{q_t\}_{t\in T})$, where $T=\gamma D(X,Y)$, such that $A_t=C^*_u(Y,d^l)$ and $B_t=C^*_u(X,d^r)$ when $t=[d]\in D(X,Y)$.

\begin{thm}
$\mathcal M=(T,M,\{M_t\}_{t\in T},\{\pi_t\}_{t\in T})$, where $T=\gamma D(X,Y)$, is a continuous field of Hilbert $\mathcal A$-$\mathcal B$-$C^*$-bimodules. 

\end{thm}
\begin{proof}
Obvious.
\end{proof}
    
\subsection{An example}

Note that $D(X,Y)$ has the smallest element $0$ given by the metric $d(x,y)=d_X(x,x_0)+1+d_Y(y_0,y)$, where $x_0\in X$, $y_0\in Y$ (different choices of these points give coarsely equivalent metrics), so Lemma \ref{lem:noncompact} does not help to check compactness. We believe that it is not compact unless either $X$ or $Y$ is bounded. Here we provide an example when it is not compact. 

Let $\mathbb N_0=\mathbb N\cup\{0\}$ with the standard metric. 
Let $Y_i=\{(y_1,y_2,\ldots,y_i,\ldots):y_i\in\mathbb N_0, y_j=0 \ \forall j\neq i\}\subset l_1$, $Y=\cup_{i\in\mathbb N}Y_i$ with the $l_1$ metric $d_Y$. To distinguish between the two copies of the same $Y$ we shall write $Y'$ for the second copy, as before.
To any metric $d\in\mathcal D(Y,Y')$ one can assign a matrix of metrics $d=(d_{ij})_{i,j\in\mathbb N}$, where $d_{ij}\in\mathcal D(Y_i,Y'_j)$ is the restriction of $d$ onto $Y_i\sqcup Y'_j$. Note that the correspondence $d\mapsto (d_{ij})_{i,j\in\mathbb N}$ is not injective.

\begin{lem}
Let $d\in\mathcal D(Y,Y')$, and let $(d_{ij})_{i,j\in\mathbb N}$ be the corresponding matrix. Then each line and each column of the matrix $([d_{ij}])_{i,j\in\mathbb N}$ cannot contain more than one non-zero element.

\end{lem}
\begin{proof}
Assume the contrary. Without loss of generality we may assume that $[d_{11}]\neq 0$ and $[d_{12}]\neq 0$. It follows from Proposition 7.1 of \cite{M1} that there exists $C>0$ and infinite sets $B_1,B_2\subset\mathbb N_0$ such that $d((n,0,0,\ldots),(n,0,0,\ldots)')<C$ for $n\in B_1$ and $d((n,0,0,\ldots),(0,n,0,\ldots)')<C$ for $n\in B_2$. Without loss of generality we may assume also that $d((0,0,0,\ldots),(0,0,0,\ldots)')<C$. The triangle inequality implies that 
$$
d_Y((n,0,0,\ldots),(m,0,0,\ldots))\geq d_Y((n,0,0,\ldots),(0,m,0,\ldots))-2C=n+m-2C.
$$
On the other hand, $d_Y((n,0,0,\ldots),(m,0,0,\ldots))=|m-n|$, so if $m>n$ then we have $m-n\geq n+m-2C$, hence $n\leq C$, and this holds for any $n\in B_1$, which is infinite --- a contradiction.
\end{proof}

If $b,d\in\mathcal D(Y,Y')$ then $[b]\leq [d]$ implies that $[b_{ij}]\leq[d_{ij}]$ for any $i,j\in\mathbb N$ (but the opposite implication need not hold).

Denote by $I$ the metric in $\mathcal D(\mathbb N_0,\mathbb N'_0)$ such that $I(n,m')=|m-n|+1$. It is easy to see that $[I]$ is the greatest element in $D(\mathbb N_0,\mathbb N'_0)$.

Let $b_k\in\mathcal D(Y,Y')$ be a metric such that $(b_k)_{ij}=\left\lbrace\begin{array}{cl}I,&\mbox{if\ }i=j\leq k;\\0,&\mbox{otherwise.}\end{array}\right.$. 

\begin{lem}
The sequence $[b_k]\in D(Y,Y')$ has no accumulation points, hence $D(Y,Y')$ is not compact.

\end{lem}
\begin{proof}
Suppose that $[d]\in D(Y,Y')$ is an accumulation point for this sequence. As $U_{[d]}$ is a neighborhood of $[d]$, $[b_k]\in U_{[d]}$ for infinitely many $k$'s. Let $k_0$ be one of these, then $[d_{ij}]=\left\lbrace\begin{array}{cl}a_i,&\mbox{if\ }i=j\leq k_0;\\0,&\mbox{otherwise,}\end{array}\right.$ where $a_i\in D(\mathbb N_0,\mathbb N'_0)$. Take another neighborhood of $[d]$, $V_c$, where $c=[b_{k_0+1}]$. Then $[b_k]\geq c$, hence $[b_k]\notin V_c$, when $k\geq k_0+1$ --- a contradiction with $[d]$ being an accumulation point.  
\end{proof}

Here is an example of a point in the corona of $D(Y,Y')$. Define $\varphi(\chi_a)=\lim_{n\to\infty}\chi_a(b_n)$, $a\in D(Y,Y)$. Note that $\chi_a(b_n)=1$ if $a$ has a non-zero element outside the diagonal, or if $a_{ii}\neq 0$ for some $i>n$, so for any $a\in D(Y,Y')$ there exists $N\in\mathbb N$ such that $\chi_a(b_n)$ equals 0 or 1 for any $n>N$, therefore, $\varphi$ is well defined. Clearly, $\varphi$ is multiplicative: $\varphi(\chi_a\chi_b)=\varphi(\chi_a)\varphi(\chi_b)$, $a,b\in D(Y,Y')$. Hence, $\varphi$ defines a point $t\in \gamma D(Y,Y')$.

\begin{lem}
$t\notin D(Y,Y')$, i.e. for any $b\in D(Y,Y')$ there exists $a\in D(Y,Y')$ such that $\varphi(\chi_a)\neq\chi_a(b)$.

\end{lem}
\begin{proof}
Suppose the contrary: $\varphi(\chi_a)=\chi_a(b)$ for any $a\in D(Y,Y')$. Then $\varphi(\chi_b)=1$, hence the matrix for $b$ is diagonal and has finitely many non-zero entries, i.e. there exists $n\in\mathbb N$ such that $b_{ii}=0$ for $i\geq n$. Take $a=b_n$. Then $\varphi(\chi_{b_n})=1$, while $\chi_{b_n}(b)=0$.  
\end{proof}

It would be interesting to know, whether the $C^*$-algebras over the points of the corona $\gamma D(X,Y)\setminus D(X,Y)$ are the uniform Roe algebras or can be more general $C^*$-algebras. In the example above there exists a metric $d\in\mathcal D(Y,Y')$ such that $A_t$ is the uniform Roe algebra for $d$.


\begin{thebibliography}{9}

\bibitem{Brodzki}
J. Brodzki, G. A. Niblo, N. Wright. \emph{Property A, partial translation structures and uniform embeddings in groups.} J. London Math. Soc. (2) {\bf 76} (2007), 479--497.

\bibitem{Dixmier}
J. Dixmier, \emph{$C^*$-Algebras}. North-Holland, Amsterdam, New York, Oxford, 1977.

\bibitem{Dupre-Fillmore}
M. J. Dupr\'e, P. A. Fillmore. \emph{Triviality theorems for Hilbert modules.} Topics in modern operator theory (Timi\c soara/Herculane, 1980),
Operator Theory: Adv. Appl., {\bf 2}, Birkha\"user, Basel--Boston, 1981, 71--79.

\bibitem{Exel}
R. Exel. \emph{Partial Dynamical Systems Fell Bundles and Applications.} Math. Surveys and Monographs, {\bf 224}, AMS, 2017.

\bibitem{M1} V. Manuilov. \emph{Metrics on doubles as an inverse semigroup.}
J. Geom. Anal., {\bf 31} (2021), 5721--5739.

\bibitem{M2} V. Manuilov, \emph{Hilbert $C^*$-modules related to discrete metric spaces.}
Russian Math., {\bf 65} (2021), № 5, 40--47. 

\bibitem{Nachbin}
L. Nachbin. \emph{Topology and Order.} Van Nostrand Math. Studies, Princeton, 1965.

\bibitem{Spakula-Willett}
J. \u Spakula, R. Willett. 
\emph{A metric approach to limit operators.} Trans. Amer. Math. Soc. {\bf 369} (2017), 263--308.

\bibitem{TRO}
 H. Zettl. \emph{A characterization of ternary rings of operators.} Adv. Math. {\bf 48} (1983), 117--143.


\end{thebibliography}
\end{document}